\documentclass{amsart}
\usepackage{mathrsfs}
\usepackage{color}
\usepackage{bm}
\usepackage{amsfonts,amssymb}
\usepackage{dsfont}
\usepackage{extarrows}
\usepackage{amsmath}
\usepackage{enumerate}
\usepackage{amscd}
\usepackage[all]{xy}
\usepackage[pagebackref,colorlinks]{hyperref}

\newtheorem{definition}{Definition}[section]

\newtheorem{corollary}{Corollary}[section]

\newtheorem{theorem}{Theorem}[section]

\newcommand{\bzm}{\begin{proof}[\indent\biaosong证明]}
\newcommand{\ezm}{\end{proof}}

\theoremstyle{remark}

\begin{document}

\title[A Kobayashi and Bergman complete domain]{A Kobayashi and Bergman complete domain without bounded representations}
\date{}
\author[N. Shcherbina]{Nikolay Shcherbina}
\address{Nikolay Shcherbina: Department of Mathematics, University of Wuppertal, 42119, Wuppertal, Germany}
\email{shcherbina@math.uni-wuppertal.de}
\author[L. Zhang]{Liyou Zhang}
\address{Liyou Zhang: School of Mathematical Sciences, Capital Normal University, 100048, Beijing, P. R.  China}
\email{zhangly@cnu.edu.cn}

\begin{abstract}
We construct an unbounded strictly pseudoconvex Kobayashi hyperbolic and complete domain in $\mathbb{C}^2$, which also possesses complete Bergman metric, but has no nonconstant bounded holomorphic functions.
\end{abstract}
\maketitle

{A question of existence of the (complete) Bergman metric on a given complex manifold $X$ is one of the important problems of  complex analysis.} When $X=\Omega$ is a bounded domain in $\mathbb{C}^n$, it is well-known that the Bergman metric {on $\Omega$} exists. If $\Omega$ is further assumed to be bounded pseudoconvex with $\mathcal{C}^1$-smooth boundary, then {the  Bergman completeness of $\Omega$ was proved by Ohsawa in \cite{Ohs1981} using} the   {celebrated Kobayshi's criterion \cite{Kobayashi} (Here the Bergman completeness means that $\Omega$ is a complete metric space with respect to the Bergman metric).} Later on, {several} articles appeared concerning the Bergman completeness for bounded pseudoconvex domains in $\mathbb{C}^n$ (see \cite{ChZh2000, Zwonek} and references therein). Among them, it is worth to mention the papers of Blocki-Pflug \cite{BP} and Herbort \cite{Her} were independently a long-standing question of Bergman completeness for any bounded hyperconvex domain has been solved (by a hyperconvex domain we mean here that {it has a bounded continuous plurisubharmonic exhaustion function). A few years later this result was further generalized by Chen \cite{Chen2003} to the case of hyperconvex manifolds.}

When $X$ is an unbounded domain, or, {more generally,} a complex manifold, {then not so much} is known for the existence of the (complete) Bergman metric except the early work of Greene-Wu \cite{GW}, which asserts that a simply-connected complex manifold possesses a complete Bergman metric if it carries a complete K\"ahler metric with holomorphic sectional curvature {which is} negatively pinched. It is a bit surprising that such lower negative bound was dropped by Chen and Zhang in \cite{ChZh2002}, where the main ingredient used in the proof is the pluricomplex Green function (see Section \ref{Completeness of the Bergman metric} below for the definition)  and the $L^2$-method for the $\bar{\partial}$-equation. In particular, they proved that a Stein manifold $X$ possesses a Bergman metric, provided that $X$ carries a bounded continuous strictly plurisubharmonic function. Recently, a new characterization {for} the existence of the Bergman metric on unbounded domains was given by Gallagher, Harz and Herbort in \cite{GHH}. It was proved there that a pseudoconvex domain with empty core (see Section \ref{Existence of the Bergman metric} below for the definition) possesses a Bergman metric.

Some other {conditions (for certain unbounded $X$) which are sufficient for possessing} a (complete) Bergman metric are also scattered in the literatures, see for examples \cite{AGK}, \cite{ChenKamimotoOhsawa04}, \cite{PZ2005}, \cite{Sh2019} et al.

\vspace{2mm}

A complex manifold $X$ is called (Kobayashi) hyperbolic if its Kobayashi pseudodistance ${\kappa}_X$ is a distance {(see Section \ref{Kobayashi completeness}  below for definitions)}. We say that $X$ is complete hyperbolic, if $(X, {\kappa}_X)$ is a complete metric space. For example, all bounded domains in $\mathbb{C}^n$ are hyperbolic, while complex manifolds containing entire curves are not. {The question of hyperbolicity has been intensively studied in the case of compact complex manifolds. Nevertheless,} there are many interesting and quite long standing conjectures in that context which are still open in their complete form, such as the Kobayashi conjecture and Green-Griffiths-Lang conjecture \cite{Demailly-Kob}.

In the case of open complex manifolds, the questions related to the Kobayashi hyperbolicity are rather different. There are various characterizations of the hyperbolicity. For examples, an analytic description due to Sibony \cite{Sibony1981} says that a complex manifold $X$ carrying a bounded strictly plurisubharmonic function is Kobayashi hyperbolic. In \cite{Abate1993}, Abate proved that a complex manifold $X$ is Kobayashi hyperbolic if and only if the space of holomorphic maps from the unit disc $\Delta$ to $X$ is relatively compact (with respect to the compact-open
topology) in the space of continuous maps from $\Delta$ into the one point compactification $X^*$ of $X$. Later, Gaussier \cite{Gaussier1999} {gave sufficient conditions for hyperbolicity of an unbounded domain}
in terms of the existence of peak and antipeak functions at infinity.  In \cite{NP2005}, Nikolov and Pflug {found} some conditions at infinity which guarantee the hyperbolicity of unbounded domains. {They also obtained} a characterization of hyperbolicity in terms of asymptotic behavior of the Lempert function. Recently, Gaussier and Shcherbina \cite{GSh2019} gave a new sufficient condition for Kobayashi hyperbolicity of unbounded domains in $\mathbb{C}^n$  using a concept of strong antipeak plurisubharmonic function at infinity. For more information on the progress in this direction (conditions which characterize (complete) Kobayashi hyperbolic (open) manifolds) we refer the reader to the survey paper of Gaussier \cite{Gaussier2017}.

In the present paper, in contrast to the situation treated by Chen in \cite{Chen2003} and Sibony in \cite{Sibony1981} (when the existence of a bounded continuous strictly plurisubharmonic function was assumed), we consider both the Bergman completeness and the Kobayashi hyperbolicity for unbounded domains in $\mathbb{C}^n$ which have neither bounded smooth strictly plurisubharmonic functions nor nonconstant bounded holomorphic functions. More precisely, we prove here the following result.

\vspace{0.3cm}
\noindent
{\bf Main Theorem.}
{{\em There exists an unbounded strictly pseudoconvex domain $\mathfrak{A}\subset\mathbb{C}^2$ with smooth boundary which has the following properties:
 \begin{enumerate}
	\item $\mathfrak{A}$ possesses a complete Bergman metric.
	\item $\mathfrak{A}$ is Kobayashi hyperbolic and complete.
	\item $\mathfrak{A}$ possesses neither nonconstant bounded holomorphic functions, nor continuous bounded strictly plurisubharmonic functions.
	\item  $\mathfrak{A}$ is not Carath\'eodory hyperbolic.
	\item  The core $\mathfrak c(\Omega)$ of $\Omega$ is nonempty, but the core $\mathfrak c'(\Omega)$ is empty.
\end{enumerate}}
{\rm (See Section \ref{Existence of the Bergman metric} below for the definitions of the cores $\mathfrak c(\Omega)$ and $\mathfrak c'(\Omega)$).}


\vspace{0.3cm}
The construction of $\mathfrak{A}$ is motivated by \cite{GSh2019}, where a Kobayashi hyperbolic Model domain with a nonempty core has been constructed. Using a characterization of the Bergman space in terms of the core (see \cite[Remark 7(b)]{GHH}), we show the existence of the Bergman metric on $\mathfrak{A}$. For the completeness of this metric, we apply a criterion given by Chen \cite[Theorem 1.1]{Chen2013}, which uses} the asymptotic behavior at infinity of the volumes of sublevel sets of the pluricomplex Green function. The completeness of Kobayashi metric follows from a criterion by Nikolov and Pflug \cite[Proposition 3.6]{NP2005}, which identifies the Kobayashi completeness and the local Kobayashi completeness in the case when there are no Cauchy sequences (w.r.t. Kobayashi metric) converging to infinity. For the nonexistence of bounded holomorphic functions on $\mathfrak{A}$, we use an argument similar to the Liouville type theorem proved in \cite[Theorom 2.2]{HST-JGA}, which says that any continuous bounded plurisubharmonic function defined in a neighborhood of a Wermer type set $\mathcal{E} \subset \mathfrak{A}$ is constant on $\mathcal{E}$.

\section{Construction of a special Wermer type set in $\mathbb{C}^2$} \label{Construction of a special Wermer type set}

Let $\{a_n\}$ be the enumeration of points running through the set $\mathbb{Z}+i\mathbb{Z} \subset \mathbb{C}_z$ as follows:
$(0,0)\to (1,0)\to (1,1)\to (0,1) \to (-1,1) \to (-1,0) \to (-1,-1) \to (0,-1)\to (1,-1) \to (2,-1) \to (2,0)\to \cdots.$

Let $\{\varepsilon_n\}$ be a decreasing sequence of positive numbers converging to zero very fast that will be further specified later. Then, as in \cite{HST}, for each $n\in \mathbb{N}$, we consider the set
\begin{align}\label{eq1}
E_n:=\left\{(z,w)\in \mathbb{C}^2: w=\varepsilon_1\sqrt{z-a_1} + ... + \varepsilon_n\sqrt{z-a_n}\right\}.
\end{align}
By definition, $\sum_{k=1}^{n}\varepsilon_k\sqrt{z-a_k}$ is a multi-valued function that takes $2^n$ values at each point $z\in\mathbb{C}$
(counted with multiplicities). Therefore, for each $z \in \mathbb{C} \setminus (\mathbb{Z}+i\mathbb{Z})$ we can locally choose single-valued functions $w_1^{(n)}(z), w_2^{(n)}(z), \cdots, w_{2^n}^{(n)}(z)$
such that

$$\sum_{k=1}^{n}\varepsilon_k\sqrt{z-a_k} = \left\{w_k^{(n)}(z): k=1,2, ... , 2^n \right\}.$$

\noindent
For every  $n\in \mathbb{N}$, we define a function $P_n: \mathbb{C}^2 \to \mathbb{C}$ as

$$P_n(z,w):=\left(w-w_1^{(n)}(z)\right)\left(w-w_2^{(n)}(z)\right)\cdots\left(w-w_{2^n}^{(n)}(z)\right).$$

\noindent
Then each $P_n$ is a well defined holomorphic polynomial in $z$ and $w$ (see for details \cite{HST}). Moreover, provided that
$\{\varepsilon_n\}$ is decreasing to zero fast enough, the
sets $E_n=\{P_n=0\}$ converge to a nonempty unbounded connected closed set
$\mathcal{E}\subset\mathbb{C}^2$, where the convergence is understood with respect to the Hausdorff metric
on each compact subset of $\mathbb{C}^2$. More precisely,
\begin{align}\label{eq2}
	\mathcal{E}=\lim_{n\to\infty}E_n=\left\{b\in\mathbb{C}^2_{z,w}: \exists \,\, b_n\in E_n, ~n=1,2, ... , \text{with}~ b=\lim_{n\to\infty}b_n\right\}.
\end{align}

Define

$$\phi_n(z,w):=\frac 1{2^n}\log\left|P_n(z,w)\right|.$$

\noindent
Then $\phi_n$ converges  uniformly on compact subsets of $\mathbb{C}^2\setminus \mathcal{E}$ to a pluriharmonic function $\phi: \mathbb{C}^2\setminus \mathcal{E}\to \mathbb{R}$, and $\lim_{(z,w)\to (z_0,w_0)}\phi(z,w)=-\infty$ for every $(z_0,w_0)\in \mathcal{E}$. In particular, $\phi$ has a unique extension to a plurisubharmonic function on the whole of $\mathbb{C}^2$ (see, for instance, \cite[Chapter I, 5.24]{Demailly-ebook}), and the set $\mathcal{E}=\{\phi=-\infty\}$
is complete pluripolar.

\section{Construction of the domain $\mathfrak{A}$} \label{Construction of the domain}

Consider first a plurisubharmonic function $\varphi$ defined on ${\mathbb C}^2$ by

$$\varphi(z,w) := \phi(z,w) + \rho(|{\rm Re}(z)|) + \rho(|{\rm Im}(z)|),$$

\noindent
where $\rho$ is the convex function constructed in the Section 2.2 of \cite{GSh2019}, and observe, that for each $t > 0$ the domain
\[U_t := \{(z, w) \in {\mathbb C}^2 : \varphi(z,w) < t\}\]
coincides with the domain $F_d$ of Lemma 2 in \cite{GSh2019} for $d = \frac{1}{2}e^t$. {The structure of the domains $F_d$, which were systematically studied in \cite{GSh2019}, will be essential for proving in Section \ref{Kobayashi completeness} below the Kobayashi completeness of our domain $\mathfrak{A}$. The crucial technical tool for this proof is the following property established in \cite{GSh2019}:

\vspace{2mm}
{\bf Property $(\mathcal F)$:} {\sl For each $d > 0$ there exists $r_0 := r_0(d) > 0$ such that the domain $F_d$ contains no holomorphic disks of radius $r > r_0$ (the last part of the statement means, more precisely, that for every holomorphic map $h : \Delta_r(0) \to F_d$ such that $\|{h'(0)}\| = 1$ one has $r \leq r_0$).}

\vspace{2mm}
Let us} now pick one of these domains, for example $U_{-1}$, and denote it (to simplify our notations) by $U$, then $U$ will be a neighborhood of the Wermer type set $\mathcal{E}$ in $\mathbb{C}^2$ which is defined by

$$U := \{(z, w) \in \mathbb{C}^2 : \varphi (z, w) < - 1 \}.$$

\noindent
Then we set
\begin{align}\label{the defining function}
\tilde{\varphi}:=-\log(-\varphi)+\tilde{\rho}(\|\zeta\|^2) ~~\text{on}~~~ \overline{U},
\end{align}
where  $\tilde{\rho}: [0, \infty) \to [0, \infty)$ is a function with the following properties:

\vspace{2mm}
(i) $\,\,\,\,\,$  $\tilde{\rho}$ is smooth strictly increasing and convex,

\vspace{2mm}
(ii) $\,\,\,$  $\tilde{\rho}(t) = t$, when $t \in [0, t_0]$, for some $t_0 > 0$,

\vspace{2mm}
(iii) $\,\,$ $\lim_{t\to \infty}{\tilde{\rho}}'(t) = \infty$.

\vspace{2mm}
\noindent
Here $\zeta = (z, w) \in \mathbb{C}^2$ and $\|\zeta\|^2:=|z|^2+|w|^2$ is the Euclidean norm. Observe that, in view of strict monotonicity and convexity of the functions $t \to -\log(-t)$ and $\tilde{\rho}$, the function $\tilde{\varphi}$ is strictly plurisubharmonic on $U$. Then, since $\tilde{\varphi}>0$ on $\partial U$, and since $\tilde{\varphi} = -\infty$ on the set $\mathcal{E}$, we conclude that, maybe after a small perturbation of the function $\tilde{\rho}$,  the domain
\begin{align}\label{the constructed domain}
\mathfrak{A} := \{x\in U: ~\tilde{\varphi}< -1\}.
\end{align}
 will be a smoothly bounded strongly pseudoconvex neighborhood of $\mathcal{E}$ such that $\overline{\mathfrak{A}} \subset U$.

Note that, by Theorem \ref{Liouville type theorem} below, we know that any bounded from above continuous plurisubharmonic function $u$ in the neighborhood of $\mathcal{E}$ is constant on $\mathcal{E}$. Therefore, the core $\mathfrak{c}(\mathfrak{A})$ of $\mathfrak{A}$ contains $\mathcal{E}$. Observe also that, for each constant $c>u\big|_{\mathcal{E}}$, the sublevel set $\{\zeta \in \mathfrak{A}: u(\zeta) \leqslant c\}$ is not relatively compact in $\mathfrak{A}$. This shows that the domain $\mathfrak{A}$ is not hyperconvex.

Now we will study the existence of the Bergman metric as well as the Bergman completeness of $\mathfrak{A}$. Note first that, since $\varphi<-1$ on $U$, the function $-\log(-\varphi)$ is well defined and, moreover, it is plurisubharmonic, since $\varphi$ is plurisubharmonic. A straightforward calculation yields

\begin{align}\label{str plurisubharmonicity}
i\partial\bar{\partial}\tilde{\varphi}=&\frac{i\partial\bar{\partial}\varphi}{-\varphi}+\frac{i\partial\varphi\wedge\bar{\partial}\varphi}{\varphi^2}+ {\tilde{\rho}}' \cdot i\partial\bar{\partial}\|\zeta\|^2+{\tilde{\rho}}''\cdot i\partial\|\zeta\|^2\wedge\bar{\partial}\|\zeta\|^2\\
\vspace{3mm}
\geqslant &{\tilde{\rho}}'(\|\zeta\|^2) \cdot i\partial\bar{\partial}\|\zeta\|^2.\nonumber
\end{align}

\noindent
Thus one can make $i\partial\bar{\partial}\tilde{\varphi}$ arbitrary large whenever $\|\zeta\|^2$ is large and $\tilde{\rho}$ grows sufficiently fast. Moreover, we can also force the volume of $\mathfrak{A}$ to be finite if $\tilde{\rho}$ goes to $+\infty$ fast enough. This insures that the holomorphic polynomials are $L^2$-integrable. In particular, the Bergman kernel of $\mathfrak{A}$ is non-degenerated.


\section{Existence of the Bergman metric} \label{Existence of the Bergman metric}

A notion of the core $\mathfrak c(\Omega)$ of a domain $\Omega\subset \mathbb{C}^n$ (or, more general, of a domain in a complex manifold $\mathcal{M}$) was introduced and intensively studied in \cite{HST-MZ}, \cite{HST-JGA} and \cite{PSh2019}. It can be defined as follows.

\begin{definition}
Let $\mathcal{M}$ be a complex manifold and let $\Omega \subset \mathcal{M}$ be a domain. Then the set

\begin{align*}
\mathfrak{c}(\Omega) := \big\{&\zeta \in \Omega : \text{every smooth plurisubharmonic function on $\Omega$ that is} \\& \text{bounded from above fails to be strictly plurisubharmonic in } \zeta \big\}
\end{align*}
is called the core of $\Omega$.
\end{definition}

\vspace{0.3cm}
Similar definition can also be given for plurisubharmonic functions from different smoothness classes.

Later on in \cite{GHH} a slightly larger class of functions defined on  a given domain $\Omega\subset \mathbb{C}^n$, was considered:
\[\mathrm{PSH}'(\Omega):=\{\phi\in \rm{PSH}(\Omega): ~\phi\not\equiv -\infty ~\text{and} ~ \nu(\phi, \cdot)\equiv 0\},\]
where $\rm{PSH}(\Omega)$ denotes the family of plurisubharmonic functions in $\Omega$ and $\nu(\phi, \zeta_0)$ denotes the Lelong number of $\phi$ at $\zeta_0$, i.e.
\[\nu(\phi, \zeta_0):=\liminf_{\zeta \to \zeta_0}\frac{\phi(\zeta)}{\log|\zeta - \zeta_0|}.\]

In order to formulate a general sufficient condition for the infinite dimensionality of the Bergman space of $\Omega$, they introduced the following notion of the core $\mathfrak c'(\Omega)$ of $\Omega$ (which is a slight modification of the notion of the core $\mathfrak c(\Omega)$ given above).

\begin{definition}
Let $\Omega$ be a domain in $\mathbb{C}^n$. Then the core $\mathfrak c'(\Omega)$ of $\Omega$ is defined by
\begin{align*}
\mathfrak c'(\Omega):=\{&\zeta\in \Omega:~ \text{every}~ \phi \in \mathrm{PSH}'(\Omega)~ \text{that is bounded from }\\
&\text{above fails to be strictly plurisubharmonic at}~ \zeta \}.
\end{align*}
\end{definition}

The following criterion for the existence of the Bergman metric on $\Omega$, in terms of the core $\mathfrak c'(\Omega)$, has been formulated in \cite[Remark 7 (b)]{GHH}.

\begin{theorem}\label{Existence of the Bergman metric}
Every pseudoconvex domain $\Omega\subset\mathbb{C}^n$ with the empty core $\mathfrak c'(\Omega)$ possesses a Bergman metric.
\end{theorem}

\vspace{0.3cm}
If we now consider the domain $\mathfrak{A}$ defined in (\ref{the constructed domain}) and observe that the function $\tilde{\varphi}$ defined by (\ref{the defining function}) is negative strictly plurisubharmonic on this domain and has zero Lelong numbers (this easily follows from the fact that  $h(t):=-(1/t)\log(-t)$ is an increasing convex function in $t$ and the fact that $\lim_{t\to -\infty}h(t)=0$, see for details \cite[Remark 7 (c)]{GHH}), then, applying Theorem \ref{Existence of the Bergman metric} to the domain $\mathfrak{A}$ on the place of $\Omega$, we conclude that the Bergman metric on $\mathfrak{A}$ exists.

\vspace*{0.3cm}
\noindent
{\bf Remark 1.} Since the restriction of $\, i\partial\bar{\partial}{\varphi} \,$ to each vertical line ${\mathbb C}_{z_0, w} := \{(z, w) \in {\mathbb C}^2 : z = z_0 \}$ is spread over the Cantor set $\mathcal{E}_{z_0} := \mathcal{E} \cap {\mathbb C}_{z_0, w}$, one can prove that the Lelong numbers of the strictly plurisubharmonic function $\varphi+\tilde{\rho}(\|\zeta\|^2) + C\|\zeta\|^2$, $C > 0$, will also be identically equal to zero. Hence, we can use this function instead of the function $\tilde{\varphi}$ in the definition of the domain $\mathfrak{A}$ and in all the arguments later on. But the proof of the fact that the Lelong numbers of $\tilde{\varphi}$ are equal to zero is more elementary, that is why we have decided to use it here.

\section{Completeness of the Bergman metric} \label{Completeness of the Bergman metric}

The main argument which we will use to prove Bergman completeness of the domain $\mathfrak{A}$ (see the Case 2 below) follows closely the arguments presented in the proof of Theorem 1.2 in \cite{Chen2013}. For the reader's convenience we give them here in details.

First we recall the notion of the pluricomplex Green function on an open set $\Omega\subset\mathbb{C}^n$ with logarithmic pole at $\zeta_0 \in \Omega$:
\[g_{\Omega}(\zeta,\zeta_0):=\sup\{u(\zeta)\}, \quad \zeta \in \Omega,\]
where the supremum is taken over all negative plurisubharmonic functions $u$ on $\Omega$ such that $u(\zeta) \leqslant \log\|\zeta-\zeta_0\|+O(1)$ near $\zeta_0$.

For each $a>0$, set \[A_{\Omega}(\zeta_0,a):=\{\zeta \in \Omega: ~~ g_{\Omega}(\zeta,\zeta_0) \leqslant -a \}.\] The following criterion for the Bergman completeness is given in \cite[Theorem 1.1]{Chen2013}.

\begin{theorem}\label{Bergman completeness}
If a Stein manifold $\Omega$ possesses the Bergman metric, then it
is Bergman complete provided the following condition is satisfied:

\vspace{0.3cm}
For any infinite sequence of points $\{\zeta_k\}$ in $\Omega$, without accumulation points
in $\Omega$, there are a subsequence $\{\zeta_{k_j}\}$, a number $a>0$ and a continuous volume
form $dV$ on $\Omega$, such that for any compact subset $K \subset \Omega$, one has
\[\int_{K\cap A_\Omega(\zeta_{k_j},a)}dV\to 0 \quad \text{as}~ j\rightarrow \infty.\]

\end{theorem}

\vspace{0.1cm}
Using this criterion we will now prove that Bergman metric on $\mathfrak{A}$ is complete.

\vspace{2mm}
Indeed, let $\{\zeta_k\}$ be an infinite sequence of points in $\mathfrak{A}$ without accumulation points in $\mathfrak{A}$.
Then we have two possibilities:

\vspace{0.3cm}
\noindent
{\bf Case 1.} {\it $\{\zeta_k\}$ admits an accumulation point on $\partial \mathfrak{A}$.}

\vspace{0.2cm}
In this case, in view of strict pseudoconvexity of $\mathfrak{A}$, a standard localization argument shows that the Bergman metric is complete. This is definitely known for experts, but for completeness of the presentation we include the sketch of the proof here.

Let $p\in\partial\mathfrak{A}$ be an accumulation point of $\{\zeta_k\}$ and let $U$ be a neighborhood of $p$. Since $\mathfrak{A}$ is strictly pseudoconvex and, hence, locally has a strictly plurisubharmonic defining function, there is a neighborhood $V\Subset U$ of the point $p$ and a constant $a>0$, such that for every point $\zeta_0 \in V$, one has $A_{\mathfrak{A}}(\zeta_0, a)\subset U\cap\mathfrak{A}$. We fix an arbitrary point $\zeta_0 \in V$ and consider a function $f\in \mathcal{O}(U\cap\mathfrak{A})$ which has the properties $f(\zeta_0)=0$ and $\|f\|^2_{L^2(U\cap\mathfrak{A})}=1$. Then we set
\begin{align}\label{eqq0}
\alpha:=\bar{\partial}(\chi\circ g_{\mathfrak{A}})\cdot f,
\end{align}
where $g_{\mathfrak{A}}$ is the pluricomplex Green function of $\mathfrak{A}$ with a pole at $\zeta_0$ (written in what follows as $g$ for simplicity) and $\chi : (-\infty, +\infty) \to [0, 1]$ is a smooth cut-off function such that $\chi(t)\equiv 1$ for $t\leqslant -2a$ and $\chi(t)\equiv 0$ for $t\geqslant -a$. Observe that $\alpha$ is a closed $(0,1)$-form which can be extended by 0 to the whole of $\mathfrak{A}$.

Next we consider the functions
\begin{align}\label{eqq1}
\psi:=-\log(-g)
\end{align}
and
\begin{align}\label{eqq2}
\varphi:=(2n+2)g.
\end{align}
Since $g$ is a negative plurisubharmonic function on $\mathfrak{A}$,
a direct calculation (see also the estimate (\ref{str plurisubharmonicity}) above with $\tilde{\rho} \equiv 0$) easily shows that
\begin{align}\label{eqq3}
i\partial g\wedge\bar{\partial}g \leq g^2 i\partial\bar{\partial}\psi.
\end{align}
Since, by (\ref{eqq0}), $\alpha = \bar{\partial}(\chi\circ g)\cdot f = \chi'\circ g \cdot \bar{\partial}g \cdot f$, it follows from (\ref{eqq3}) that
\begin{align}\label{eqq4}
|\alpha|^2=|\chi'\circ g|^2 \cdot i\partial g\wedge\bar{\partial}g \cdot |f|^2 \leq |\chi'\circ g|^2 \cdot g^2 \cdot |f|^2 \cdot i\partial\bar{\partial}\psi = H \cdot i\partial\bar{\partial}\psi,
\end{align}
where $H := |\chi'\circ g|^2 \cdot g^2 \cdot |f|^2$. Then, by the celebrated Donnelly-Fefferman's existence theorem (see for example \cite[Theorem 3.1]{Blocki2005}), and in view of (\ref{eqq1}) and (\ref{eqq4}), there exists $u\in L^2_{loc}(\mathfrak{A})$ such that $\bar{\partial}u=\alpha$ and the following estimate holds
\begin{align}\label{Donnelly-Fefferman estimate}
\int_{\mathfrak{A}}|u|^2e^{-\varphi}\leqslant 16 \cdot \int_{\mathfrak{A}}H \cdot e^{-\varphi}.
\end{align}

\noindent
From the definition of $H$ and the fact that $\|f\|^2_{L^2(U\cap\mathfrak{A})}=1$ we can conclude by (\ref{Donnelly-Fefferman estimate}) that
\[
\int_{\mathfrak{A}}|u|^2e^{-\varphi} \leqslant 16 \int_{\mathfrak{A}}|\chi'\circ g|^2 g^2|f|^2 e^{-\varphi} \leqslant 16 \cdot C(a) \cdot 4 a^2 \int_{A_{\mathfrak{A}}(\zeta_0, a)\setminus A_{\mathfrak{A}}(\zeta_0, 2a)}|f|^2 \leqslant \tilde C(a),
\]
where $C(a)$ and $\tilde C(a)$ are some constants that depend on $a$ only (for ${\mathfrak{A}}$, $U$ and $V$ being fixed). It follows now from the fact that $g(\zeta, \zeta_0)-\log|\zeta-\zeta_0|$ is locally bounded and, hence, in view of (\ref{eqq2}), from the fact that the integrability of $|u|^2e^{-\varphi} = |u|^2e^{-(2n+2)g}$ is the same as the integrability of $|u|^2|\zeta-\zeta_0|^{-(2n+2)}$, that $u(\zeta_0)=0$ and $\frac{\partial u}{\partial \zeta}(\zeta_0)=0$. Therefore, if we set
\begin{align}\label{eqq5}
F:=(\chi\circ g)f-u,
\end{align}
we will get a function $F\in \mathcal{O}(\mathfrak{A})$ such that
\begin{align}\label{eqq5a}
F(\zeta_0)=f(\zeta_0) ~~~\text{and}~~~~~~~ \frac{\partial F}{\partial \zeta}(\zeta_0)=\frac{\partial f}{\partial \zeta}(\zeta_0).
\end{align}
By the estimate (\ref{Donnelly-Fefferman estimate}), we have that
\[\|F\|_{L^2(\mathfrak{A})}\leqslant \|(\chi\circ g)f\|^2_{L^2(\mathfrak{A})}+\|u\|^2_{L^2(\mathfrak{A})}\leqslant (1+\tilde C(a))\|f\|^2_{L^2(U\cap\mathfrak{A})}=:\tilde{\tilde{C}}. \]
Recall that the Bergman metric is defined by
\[ds^2_{\mathfrak{A}}(\zeta, X):=K^{-1}_{\mathfrak{A}}(\zeta)\sup \{|Xf|^2: f\in \mathcal{O}(\mathfrak{A}) ~ \text{with}~  f(\zeta)=0~ \text{and}~ \|f\|_{L^2(\mathfrak{A})}\leqslant1\},\] where $X$ is a nonzero tangent vector at $\zeta$. If we assume that the function $f$ achieves the above supremum on $U\cap\mathfrak{A}$, we will get by (\ref{eqq5}) a function $F\in \mathcal{O}(\mathfrak{A})$ such that, in view of  (\ref{eqq5a}), for any $\zeta\in V\Subset U$ the following estimate holds
\[K_{\mathfrak{A}}(\zeta)\cdot ds^2_{\mathfrak{A}}(\zeta, X)\geqslant \tilde{\tilde{C}}^{-1}|XF|^2=\tilde{\tilde{C}}^{-1}|Xf|^2=\tilde{\tilde{C}}^{-1}K_{U\cap\mathfrak{A}}(\zeta)\cdot ds^2_{U\cap\mathfrak{A}}(\zeta, X).\]
We can conclude now from the trivial inequality $K_{\mathfrak{A}}(\zeta)\leqslant K_{U\cap\mathfrak{A}}(\zeta)$ that
$$ds^2_{\mathfrak{A}}(\zeta, X)\geqslant \tilde{\tilde{C}}^{-1} ds^2_{U\cap\mathfrak{A}}(\zeta, X),$$
for all $\zeta\in V$. This completes the proof of the localization property for the Bergman metric and shows the completeness of the Bergman metric at the finite points of $\partial \mathfrak{A}$.

\vspace{0.3cm}
\noindent
{\bf Case 2.} {\it $\zeta_k\to \infty$ as $k\to \infty$.}

\vspace{0.2cm}
Consider a smooth cut-off function $\chi$ on $\mathbb{C}^2$ such that $\chi(\zeta) = 1$, when $\|\zeta\| \leqslant \frac{1}{2}$ and $\chi(\zeta) = 0$, when $\|\zeta\| \geq 1,$ and then for every $\delta > 0$ define the auxiliary function \[u_{\delta,k}(\zeta) := \delta\tilde\varphi(\zeta) + \chi(\zeta-\zeta_k)\log\|\zeta-\zeta_k\|.\]
 Notice that there exists a constant $C_1 > 0$ such that
 \[i\partial\bar\partial\left(\chi(\zeta)\log\|\zeta\|^2\right) \geqslant -{C_1}i\partial\bar\partial\|\zeta\|^2.\]
 Let $K$ be any compact subset in $\mathfrak{A}$. Observe that for each $\delta>0$ there exists a positive integer $k_{\delta}$ such that for every $k>k_{\delta}$ one has that $B(\zeta_k) \cap K=\emptyset$ (here $B(\zeta_k)$ denote the ball of radius $1$ with center at $\zeta_k$) and, moreover, that the function $u_{\delta,k}$ is a negative and plurisubharmonic on $\mathfrak{A}$.
We can insure plurisubharmonicity of $u_{\delta,k}(\zeta)$ for large enough $k_\delta$ in the following way:
 \begin{itemize}
   \vspace{0.2cm}	
   \item On the set $\mathfrak{A} \setminus \overline{B(\zeta_k)}$ plurisubharmonicity of $u_{\delta,k}(\zeta)$ is clear, since on this set $\chi(\zeta - \zeta_k) \equiv 0$.
   \vspace{0.2cm}
   \item On the set $B(\zeta_k)$ one has $$i\partial\bar\partial u_{\delta,k} \geqslant \delta\cdot i\partial\bar{\partial}\tilde{\varphi}- {C_1}i\partial\bar\partial\|\zeta\|^2 \geqslant \left(\delta \cdot  {\tilde{\rho}}'(\|\zeta\|^2)-{C_1}\right)i\partial\bar\partial\|\zeta\|^2.$$
   This, in view of our assumptions that $\lim_{t\to \infty}{\tilde{\rho}}'(t) = \infty$ and that $\zeta_k\to \infty$ as $k\to \infty$, implies plurisubharmonicity of $u_{\delta,k}(\zeta)$ on $B(\zeta_k)$ for all large enough  $k_\delta$.
 \end{itemize}

\vspace{0.2cm}
\noindent
Observe now that, since $u_{\delta,k}$ is a negative plurisubharmonic function with (at least) a logarithmic pole at $\zeta_k$, it is included into the family of functions which defines the pluricomplex Green function $g_{\mathfrak{A}}(\zeta,\zeta_k)$. By the definition of the pluricomplex Green function, for large enough $k$ we have
 \[g_{\mathfrak{A}}(\zeta,\zeta_k) \geqslant u_{\delta,k}(\zeta)=\delta\cdot\tilde\varphi \quad \text{for} ~ \zeta\in K\]
 and, hence, also
 \[K\cap A_{\mathfrak{A}}(\zeta_k, a) \subset \{\zeta \in K: \tilde\varphi \leq -a/\delta\}.\]

\noindent
Since for every $\varepsilon>0$ there exists $\delta > 0$ so small that the volume of the set $\{\zeta\in K: \tilde\varphi<-a/\delta\}$ is less than $\varepsilon$, we conclude by the above arguments, that there exists $k_\delta$ such that
\[\int_{K \cap A_{\mathfrak{A}}(\zeta_k, a)}dV<\varepsilon\] for all $k>k_\delta$.
By Theorem \ref{Bergman completeness}, the Bergman metric is complete.

\section{Kobayashi completeness} \label{Kobayashi completeness}

 Let $\Omega$ be an arbitrary domain in $\mathbb{C}^n$. Recall that the Kobayashi pseudometric $\kappa_{\Omega}$ at a point $(\zeta, v)\in \Omega \times \mathbb{C}^n$ is defined as
\begin{align}\label{def: Kobayashi metric}
\kappa_{\Omega}(\zeta, v) := \inf\{\frac 1r: \exists h\in \mathcal{O}(\Delta, \Omega)~\text{with}~ h(0)=\zeta ~\text{and}~ h'(0)=rv\},
\end{align}
where $\Delta:=\{z\in \mathbb{C}: |z|<1\}$ is the unit disc in $\mathbb{C}$ and $\mathcal{O}(\Delta,\Omega)$ denotes the set of all  holomorphic maps from $\Delta$ to $\Omega$.

For any two point $\zeta_1, \zeta_2 \in \Omega$, the Kobayashi pseudodistance is defined by
\begin{align}\label{def: Kobayashi distance}
\kappa_{\Omega}(\zeta_1,\zeta_2)=\inf\left\{\int_0^1\kappa_{\Omega}(\gamma(t),\gamma'(t))dt\right\},
\end{align}
where the infimum is taken over all piecewise $\mathcal{C}^1$ curve $\gamma: [0, 1]\to \Omega$ connecting $\zeta_1$ and $\zeta_2$.

The domain $\Omega$ is called Kobayshi hyperbolic if $\kappa_{\Omega}$ is a distance. $\Omega$ is said to be Kobayashi complete (abbr. $\kappa$-complete) if any $\kappa_{\Omega}$-Cauchy sequence $\{\zeta_j\}_{j\in \mathbb{N}}$ converges to a point $\zeta_0\in \Omega$ with respect to the Euclidean topology of $\Omega$.

When $\Omega$ is a bounded domain in $\mathbb{C}^n$, it is well-known that $\Omega$ is $\kappa$-complete if and only if $\Omega$ is locally $\kappa$-complete, i.e., for any boundary point $a \in \partial \Omega$, there exists a bounded neighborhood $U$ of $a$ such that every connected component of $\Omega \cap U$ is $\kappa$-complete (see, for example, \cite[Theorem 7.5.5]{PJ1993}).

For unbounded domains which are (possibly) not biholomorphic to bounded domains in $\mathbb{C}^n$, Nikolov and Pflug gave a criterion of the $\kappa$-completeness, by introducing the following notions of $\kappa$-points and $\kappa'$-points.

\begin{definition} A point $a\in \partial \Omega$ is called a ${\it \kappa}$-{\it point} for $\Omega$ if $\lim\limits_{\zeta \to a}\kappa_{\Omega}(\zeta,\xi)=\infty$ for any fixed $\xi \in \Omega$.
\end{definition}

\begin{definition} A point $a \in \partial \Omega$  is called a ${\it \kappa}'$-{\it point} for $\Omega$ if there is no $\kappa_\Omega$-Cauchy sequence converging to $a$.
\end{definition}

\vspace{0.3cm}
It is clear that any $\kappa$-point is a $\kappa'$-point. Nikolov and Pflug (see \cite[Proposition 3.6]{NP2005}) have proved the following theorem for the $\kappa$-completeness.

\begin{theorem}
Let $\Omega$ be an open set in $\mathbb{C}^n$. Assume that $\infty$ is a $\kappa'$-point if $\Omega$ is unbounded. Then the following conditions are equivalent:
 \begin{enumerate}
 \item $\Omega$ is $\kappa$-complete.
 \item Any finite boundary point of $\Omega$ admits a neighborhood $U$ such that $\Omega \cap U$ is $\kappa$-complete.
  \item Any finite boundary point of $\Omega$ is a $\kappa'$-point.
  \item Any boundary point of $\Omega$ is a $\kappa$-point.
\end{enumerate}
\end{theorem}

\vspace{0.3cm}
We will now use the above theorem to prove that the domain $\mathfrak{A}$ is $\kappa$-complete.

\vspace{2mm}
Observe first that, by our definition of the domain $\mathfrak{A}$ (which is given in Section \ref{Construction of the domain}), one has that
\[\mathfrak{A} \subset U_{-1} = F_d\]
with $d = {\frac{1}{2}e^{-1}}$. By Property ($\mathcal F$) stated in Remark 3 of \cite{GSh2019} and restated in Section 2 above, we know that the domain $U_{-1} = F_{\frac{1}{2}e^{-1}}$, and hence also a smaller domain $\mathfrak{A}$, contains only discs of uniformly bounded size, say $r_0$. This implies that $\mathfrak{A}$ is Kobayashi hyperbolic. Since the domain $\mathfrak{A}$ is strictly pseudoconvex at each finite boundary point, it follows that $\mathfrak{A}$ is $\kappa$-complete at these points (see, for example, the arguments of Lemma 2.1.1 and Lemma 2.1.3 in \cite{Gaussier1999}). Hence, for proving that $\mathfrak{A}$ is $\kappa$-complete we only need to check that $\infty$ is a $\kappa'$-point of $\mathfrak{A}$.

Now suppose that $\{\zeta_j\}$ is a $\kappa$-Cauchy sequence converging to $\infty$. Then, by the definitions (\ref{def: Kobayashi metric}), (\ref{def: Kobayashi distance}) of the Kobayashi pseudodistance, and in view of the boundedness by $r_0$ of the size of holomorphic disks contained in $\mathfrak{A}$, we see that
\[\frac{1}{r_0}d_E(\zeta_k, \zeta_l) \leq \kappa_{\mathfrak{A}}(\zeta_k,\zeta_l)\]
for all $k, l \in \mathbb{N}$, where $d_E(\zeta, \xi)$ denotes the Euclidean distance between the points $\zeta, \xi \in {\mathbb{C}}^2$. The last inequality obviously implies that $\{\zeta_j\}$ is also a Cauchy sequence with respect to the Euclidean distance, and hence can not converge to $\infty$. This shows that $\infty$ is a $\kappa'$-point of $\mathfrak{A}$, and, therefore, proves the Kobayashi completeness of $\mathfrak{A}$.

\section{Nonexistence of nonconstant bounded holomorphic functions} \label{Nonexistence of nonconstant bounded holomorphic functions}

In this section we will show that the only bounded holomorphic functions defined in $\mathfrak{A}$ are constants. For proving this we will need the following version of the Liouville type theorem which was proved in \cite[Theorem 2.2]{HST-JGA} for a slightly different Wermer type set. The argument provided there can also be applied to the set $\mathcal{E}$ considered in our paper. But, since for this set there is a bit simpler proof, we will present it here for the reader convenience.

\begin{theorem}\label{Liouville type theorem}
 Let $\phi$ be a continuous plurisubharmonic function defined on an open neighborhood $V \subset \mathbb{C}^2$ of the Wermer type set $\mathcal{E}$. If $\phi$ is bounded from above, then $\phi \equiv C$ on $\mathcal{E}$ for some constant $C\in \mathbb{R}$.
 \end{theorem}

\begin{proof}

 By construction of the Wermer type set (see Section \ref{Construction of a special Wermer type set}), $\mathcal{E}$ is locally the limit in the Hausdorff distance of analytic sets $E_n$ and, therefore, the complement of $\mathcal{E}$ is pseudoconvex. Due to a theorem of Slodkowski, we know that $\mathcal{E}$ is an analytic multivalued function (see the definition of analytic multivalued functions and Slodkowski's theorem in \cite[pages 15-16]{Aupetit83}).

Now, for a given set $F \subset \mathbb{C}^2_{z,w}$ and any point $z_0\in \mathbb{C}_z$, let us define the vertical slice $F(z_0)$ of the set $F$ by
$$F(z_0):=\{w\in \mathbb{C}_w: (z_0, w)\in F\}.$$
Then for any function $\phi$, which is continuous and plurisubharmonic in a neighbourhood of the set $\mathcal{E}$, we can define the function
$$\tilde\phi(z_0):=\max_{w \in \mathcal{E}(z_0)}\phi(z_0, w).$$
It follows from part (ii) of the definition of analytic multivalued functions that $\tilde\phi$ is a subharmonic function in the complex plane $\mathbb{C}_z$. If the function $\phi$ is further assumed to be bounded from above (without loss of generality we may assume that on the set $\mathcal{E}$ one has $\phi\leqslant 0$ and $\sup \phi=0$), the standard Liouville theorem asserts that $\tilde\phi\equiv 0$ on $\mathbb{C}_z$.

The rest of the proof will be devoted to showing that even for the initial function $\phi$ one also has that $\phi\equiv 0$.
In order to do this we first define the set
$$\mathcal{A}:=\{p\in\mathcal{E}: \phi(p)=0, p_z := \pi_z(p) \notin \mathbb{Z}+ i \mathbb{Z}\},$$
where $\pi_z : \mathbb{C}^2_{z,w} \to \mathbb{C}_z$ is the canonical projection. Then we observe that for every point $p = (p_1, p_2) \in \mathcal{A}$ and every analytic disc $D_r(p) := \{(z, f(z)): z \in \Delta_r(p_1) \} \subset \mathcal{E} $ around $p$ such that $\Delta_r(p_1) \subset \mathbb{C}_z \setminus \{\mathbb{Z}+ i \mathbb{Z}\}$, $p=(p_1, f(p_1))$ and $f : \Delta_r(p_1) \to \mathbb{C}_w$ holomorphic, the function $\phi(z, f(z))$ is subharmonic on $\Delta_r(p_1) \subset \mathbb{C}_z$. Here $\Delta_r(p_1)$ denote the disc of radius $r$ centered at $p_1$. Since $\phi|_\mathcal{E} \leqslant 0$ and $\phi(p) = 0$, we conclude from the maximum principle that $\phi(z, f(z))$ is identically equal to zero on $\Delta_r(p_1)$ and, therefore, the set $\mathcal{A}$ is open in the topology defined along the leaves of $\mathcal{E}$. 

More precisely, we say that a curve $\eta : [0, 1] \to \mathcal{E}$ {\em follows the leaves of } $\mathcal{E}$ if for each $t \in [0 ,1]$ there exist $\varepsilon > 0$, $r > 0$ and a holomorphic function $f : \Delta_r(\pi_z(\eta(t))) \to \mathbb{C}_w$ such that $\Gamma(f) := \{(z ,f(z)) \in \mathbb{C}^2_{z,w} : z \in \Delta_r(\pi_z(\eta(t)))\} \subset \mathcal{E}$ and for each $t' \in (t - \varepsilon, t + \varepsilon)$ (or, $t' \in [0 , t + \varepsilon)$ for $t = 0$ and $t' \in (1 - \varepsilon, 1]$ for $t = 1$, respectively) one has that $\eta(t') \in \Gamma(f)$.

\vskip0.3cm
\noindent
{\bf Claim 1.} If $\eta : [0, 1] \to \mathcal{E}$ is a continuous curve which follows the leaves of $\mathcal{E}$ and such that $\phi(\eta(0)) \in \mathcal{A}$ and $\pi_z(\eta(t)) \notin \mathbb{Z}+ i \mathbb{Z}$ for all $t \in [0, 1]$, then $\phi(\eta(1)) \in \mathcal{A}$.

\vskip0.3cm

The claim easily follows if we cover the curve $\eta([0, 1])$ by suitable finite family of analytic discs $D_1, D_2, ... ,D_l \subset \mathcal{E}$ and apply the maximum principle to the restriction of $\phi$ to each of these discs.

Now, in analogy to the definitions (\ref{eq1}) and (\ref{eq2}) of the sets $E_n$ and $\mathcal{E}$, respectively, for each $m \in \mathbb{N}$ and each $n \geq m$ we can consider the set
$$E_{m,n}:=\left\{(z,w)\in \mathbb{C}^2: w=\varepsilon_m\sqrt{z-a_m} + ... + \varepsilon_n\sqrt{z-a_n}\right\}$$
and then define the set
$$	\mathcal{E}_m=\lim_{n\to\infty}E_{m,n}=\left\{b\in\mathbb{C}^2_{z,w}: \exists \,\, b_n\in E_{m,n}, ~n=m,m+1, ... , \text{with}~ b=\lim_{n\to\infty}b_n\right\}.$$
It is easy to see from these definitions that for each $m \in \mathbb{N}$ we have that
\begin{align}\label{eq8}
\mathcal{E}= E_m  \oplus \mathcal{E}_{m+1} := \{(z, w_1 + w_2) \in \mathbb{C}^2_{z,w}: w_1 \in E_m(z), w_2 \in \mathcal{E}_{m+1}(z) \},
\end{align}
and then, from the construction of the set $\mathcal{E}$ (more precisely, from the choice of the sequence $\{\varepsilon_n\}$ in this construction) we also conclude that, if for every compact set $K \subset \mathbb{C}_z$ and every $m \in \mathbb{N}$ we define
$$\delta_K(m) := \max_{z \in K}{\max_{w \in \mathcal{E}_{m}(z)} |w|},$$
then
\begin{align}\label{eq9}
\lim_{m \to \infty} \delta_K(m) = 0.
\end{align}

Next, for a fixed $m \in \mathbb{N}$, we define the notion of the {\it lifting} $\tilde{\gamma}$ of a curve $\gamma \subset \mathbb{C}_z$ to the set $\mathcal{E}_m$. Let $p$ be a point of $\mathcal{E}_m$ such that $\pi_z(p) \notin \mathbb{Z}+ i \mathbb{Z}$. Then for some $r > 0$ we have that $\Delta_r(\pi_z(p)) \subset \mathbb{C}_z \setminus \{\mathbb{Z}+ i \mathbb{Z}\}$ and, hence, the set $\mathcal{E}_m \cap (\Delta_r(\pi_z(p)) \times \mathbb{C}_w)$ can be represented as the union of the graphs $\{\Gamma(f_\alpha)\}_{\alpha \in \mathcal{B}}$ of holomorphic functions $\{f_\alpha : \Delta_r(\pi_z(p)) \to \mathbb{C}_w\}_{\alpha \in \mathcal{B}}$. Let $f_{\alpha(p)}$ be a function of this family such that its graph $\Gamma(f_{\alpha(\zeta)})$ contains the point $p$ and let $\gamma: [0,1]\to \mathbb{C}_z\setminus \{\mathbb{Z}+i \mathbb{Z}\}$ be a continuous curve such that $\pi_z(p)=\gamma(0)$. Then we define the lifting of the curve $\gamma$ with the initial data $f_{\alpha(p)}$ as follows:

We divide the segment $[0,1]$ into small enough segments by points $0 = t_0 < t_1 < ... < t_k =1$ and then consider a family of disks $\{\Delta_{r_s}(\gamma(t_s))\}_{1 \leq s \leq k}$ such that for each $s = 1, 2, ..., k$ one has $\Delta_{r_s}(\gamma(t_s)) \subset \mathbb{C}_z \setminus \{\mathbb{Z}+ i \mathbb{Z}\}$ and $\gamma([t_{s - 1}, t_s]) \subset \Delta_{r_s}(\gamma(t_s))$ (in particular, $\{\Delta_{r_s}(\gamma(t_s)\}_{1 \leq s \leq k}$ is a covering of $\gamma([0, 1])$). Observe that, as above, for each $s = 1, 2, ..., k$ the set $\mathcal{E}_m \cap (\Delta_{r_s}(\gamma(t_s)) \times \mathbb{C}_w)$ can be represented as the union of the graphs $\{\Gamma(f_\alpha^s)\}_{\alpha \in \mathcal{B}}$ of holomorphic functions $\{f_\alpha^s : \Delta_{r_s}(\gamma(t_s)) \to \mathbb{C}_w\}_{\alpha \in \mathcal{B}}$. Note now that, in view of the construction of the set $\mathcal{E}_m$ and the unicity theorem for holomorphic functions, there exists exactly one function $f_{\alpha(p)}^1$ in the family $\{f_\alpha^1 : \Delta_{r_1}(\gamma(t_1)) \to \mathbb{C}_w\}_{\alpha \in \mathcal{B}}$ which coincides with $f_{\alpha(p)}$ on the set $\Delta_{r_0}(\gamma(t_0)) \cap \Delta_{r_1}(\gamma(t_1))$ (without loss of generality we can assume here that $r_0 < r$, and hence the function $f_{\alpha(p)}$ is defined on the disk $\Delta_{r_0}(\gamma(t_0))$). Then we proceed inductively and, if for some $2 \leq s \leq k$ the function $\{f_{\alpha(p)}^{s - 1} : \Delta_{r_{s - 1}}(\gamma(t_{s - 1})) \to \mathbb{C}_w\}$ is already chosen, we consider the (uniquely defined) function from the family $\{f_\alpha^s : \Delta_{r_s}(\gamma(t_s)) \to \mathbb{C}_w\}_{\alpha \in \mathcal{B}}$ which coincides with $f_{\alpha(p)}^{s - 1}$ on the set $\Delta_{r_{s - 1}}(\gamma(t_{s - 1})) \cap \Delta_{r_s}(\gamma(t_s))$ and denote it by $f_{\alpha(p)}^s$. Now we can finally define the lifting $\tilde{\gamma}$ of the curve $\gamma$ to the set $\mathcal{E}_m$. Namely, for each $s = 1, 2, ..., k$ and each $\theta \in [t_{s - 1}, t_s]$ we set
$$\tilde{\gamma}(\theta) := ({\gamma}(\theta), f_{\alpha(p)}^s({\gamma}(\theta))).$$

To finish the proof of the theorem we need to show that for an arbitrary point $q \in \mathcal{E}$ one has $\phi(q) = 0$. In view of continuity of the function $\phi$, it is enough to prove this claim for points $q$ such that $\pi_z(q) \notin \mathbb{Z}+ i \mathbb{Z}$. Let us fix a point $q$ with these properties and let $p$ be a point of the set $\mathcal{A}$, i.e. $\phi(p) = 0$ and $\pi_z(p) \notin \mathbb{Z}+ i \mathbb{Z}$. Fix $R > 0$ so large that $\pi_z(p), \pi_z(q) \in \Delta_R(0)$ and set $K := \overline{\Delta}_R(0)$. Then, by property (\ref{eq9}) above, for each $n \in \mathbb{N}$ there is $m_n \in \mathbb{N}$ such that
\begin{align}\label{eq10}
\max_{z \in K}{\max_{w \in \mathcal{E}_{m}(z)} |w|} < \frac{1}{2^n}
\end{align}
for all $m \geq m_n$. Further, in view of property (\ref{eq8}), we can find points $p_n, q_n \in E_{m_n}$ such that $\pi_z(p_n) = \pi_z(p)$, $\pi_z(q_n) = \pi_z(q)$ and
\begin{align}\label{eq11}
\tilde{p}_n := (\pi_z(p_n), \pi_w(p) - \pi_w(p_n)) \in \mathcal{E}_{m_n + 1},
\end{align}
\begin{align}\label{eq12}
\tilde{q}_n := (\pi_z(q_n), \pi_w(q) - \pi_w(q_n)) \in \mathcal{E}_{m_n + 1},
\end{align}
where $\pi_w :  \mathbb{C}^2_{z, w} \to \mathbb{C}_w$ is the canonial projection. Since the set $E_{m_n} \setminus (\{\mathbb{Z}+ i \mathbb{Z}\} \times \mathbb{C}_w)$ is obviously connected, there is a continuous curve $\eta_n : [0, 1] \to E_{m_n}$ such that $\eta_n(0) = p_n$, $\eta_n(1) = q_n$ and $\pi_z(\eta_n(t)) \notin \mathbb{Z}+ i \mathbb{Z}$ for all $t \in [0, 1]$. Let $\gamma_n : [0, 1] \to \mathbb{C}_z \setminus \{\mathbb{Z}+ i \mathbb{Z}\}$ be a curve defined by $\gamma_n(t) := \pi_z(\eta_n(t))$. Consider now some initial data $f_{\alpha(\tilde{p}_n)}$ for the lifting of the curve $\gamma_n$ to the set $\mathcal{E}_{m_n + 1}$, that is, consider a number $r > 0$ and a holomorphic function $f_{\alpha(\tilde{p}_n)} : \Delta_r(\pi_z(p_n)) \to \mathbb{C}_w$ such that its graph $\Gamma(f_{\alpha(\tilde{p}_n)})$ contains the point $\tilde{p}_n$ and is contained in the set $\mathcal{E}_{m_n + 1}$. In the case when such initial data is not unique, one can just choose an arbitrary one. Then, in view of the construction above of the lifting of a curve, we will get a lifting $\tilde{\gamma}_n$ of the curve $\gamma_n$ to the set $\mathcal{E}_{m_n + 1}$ with the initial point $\tilde{\gamma}_n(0) = \tilde{p}_n$ and an endpoint $\tilde{\gamma}_n(1) =: \tilde{\tilde{q}}_n \in \mathcal{E}_{m_n + 1}(\pi_z(q))$. Now we can finally define the curve $\gamma_n^\ast : [0, 1] \to \mathcal{E} \setminus (\{\mathbb{Z}+ i \mathbb{Z}\} \times \mathbb{C}_w)$ by
$$\gamma_n^\ast(t) := (\gamma_n(t), \pi_w(\eta_n(t)) + \pi_w(\tilde{\gamma}_n(t)))$$
and observe that, by our construction and property (\ref{eq11}), one has that
\begin{equation} \label{eq13}
\begin{split}
\gamma_n^\ast(0) &= (\pi_z(\eta_n(0)), \pi_w(\eta_n(0)) + \pi_w(\tilde{\gamma}_n(0))) = (\pi_z(p_n), \pi_w(p_n) + \pi_w(\tilde{p}_n))\\
&= (\pi_z(p), \pi_w(p)) = p
\end{split}
\end{equation}
and, by property (\ref{eq12}), one also has that
\begin{equation} \label{eq14}
\begin{split}
\gamma_n^\ast(1) &= (\pi_z(\eta_n(1)), \pi_w(\eta_n(1)) + \pi_w(\tilde{\gamma}_n(1))) = (\pi_z(q_n), \pi_w(q_n) + \pi_w(\tilde{\tilde{q}}_n)\\
&= (\pi_z(q), \pi_w(q) - \pi_w(\tilde{q}_n) + \pi_w(\tilde{\tilde{q}}_n)) =: q_n^\ast.
\end{split}
\end{equation}
Since, by properties (\ref{eq13}) and (\ref{eq14}), $\gamma_n^\ast$ is a curve in $\mathcal{E} \setminus (\{\mathbb{Z}+ i \mathbb{Z}\} \times \mathbb{C}_w)$ connecting the points $p$ and $q_n^\ast$, and since, by the choice of $p$, one has that $\phi(p) = 0$, we can conclude from Claim 1 that $\phi(q_n^\ast) = 0$. But properties (\ref{eq10}), (\ref{eq12}), (\ref{eq14}) and the fact that $\tilde{\tilde{q}}_n \in \mathcal{E}_{m_n + 1}(\pi_z(q))$ imply that
$$\|q - q_n^\ast\| = |\pi_w(\tilde{q}_n) - \pi_w(\tilde{\tilde{q}}_n)| < |\pi_w(\tilde{q}_n)| + |\pi_w(\tilde{\tilde{q}}_n)| < \frac{1}{2^{n - 1}}.$$
Hence, by continuity of $\phi$, we finally have that $\phi(q) = \lim_{n \to \infty}\phi(q_n^\ast) = 0$. This concludes the proof of our Liouville type theorem.

\end{proof}

Let us now consider an arbitrary bounded holomorphic function $f$ on the domain $\mathfrak{A}$. Without loss of generality, we can assume that $|f| < 1$ on $\mathfrak{A}$. Then the following statement holds true.

\vspace{2mm}
\noindent
\textbf{Claim 2.} The restriction $f|_{\mathcal{E}}$ of the function $f$ to the Wermer type set $\mathcal{E} \subset \mathfrak{A}$ is constant.

\vspace{2mm}

\begin{proof}
Due to the last theorem, we know that $|f|\equiv r_0$ on $\mathcal{E}$ for some constant $r_0<1$. Hence, we can write $f(z,w)=r_0\exp(i\theta),$ where $\theta\in [0, 2\pi]$, in principle, might be different for different values $z$ and $w$. Consider now the function $g:=|1-f|$. Then $g$ is also bounded continuous and plurisubharmonic on some neighborhood of $\mathcal{E}$, hence, by the last theorem, $|g|$ is also constant (say $C^\ast$) on $\mathcal{E}$. A simple calculation shows that there are at most two solutions $\theta_0$ and $2\pi-\theta_0$ to the equation $|1-r_0\exp(i\theta)|= C^\ast$. It follows then from holomorphicity (and, hence, continuity) of $f$ and connectedness of $\mathcal{E}$ that the restriction $f|_{\mathcal{E}}$ of $f$ to $\mathcal{E}$ is constant. This completes the proof of Claim 2.
\end{proof}

Now we fix $z=z_0$ and consider the slice $\mathfrak{A}_{z_0} := \mathfrak{A} \cap \{z=z_0\}$. Since, by construction, the core $\mathcal{E}$ intersected with $\mathfrak{A}_{z_0}$ is a Cantor set, it follows that there is a point $w_0$ and a sequence of points $w_j$ converging to $w_0$ such that $\zeta_j = (z_0, w_j) \in \mathcal{E} \cap \mathfrak{A}_{z_0}$ for each $j \in \mathbb N$. Then, by Claim 2, there is a constant $C$ such that $f|_{\mathcal{E}} \equiv C$. In particular, we have that $f(\zeta_j) = C$ for each $j \in \mathbb N$. The standard one-dimensional identity theorem for holomorphic functions tells us now that $f = C$ on the whole slice $\mathfrak{A}_{z_0}$. Since the same argument holds true for each $z_0$, we finally conclude that $f \equiv C$ on  $\mathfrak{A}$, which completes the proof of the main statement of this Section.


\vspace{2mm}
The next statement follows directly from nonexistence of bounded holomorphic functions.

\begin{corollary}
{{ The domain $\mathfrak{A}$ has the following properties:
 \begin{enumerate}
     \item $\,\,$  $\mathfrak{A}$ can not be biholomorphic to a bounded domain.
     \item The Carath\'eodory metric on $\mathfrak{A}$ is identically equal to zero.
 \end{enumerate}}}

 \end{corollary}

\vspace{3mm}
\noindent
{\bf Acknowledgement.} {\it Part of this work was done while the first author was a visitor at the Capital Normal University (Beijing). It is his pleasure to thank this institution for its hospitality and good working conditions.}

\end{document}